\documentclass[11pt, a4paper, final]{amsart}
\usepackage[T1]{fontenc}
\usepackage{lmodern}
\usepackage{mathtools}
\usepackage[utf8]{inputenc}

\usepackage{amsfonts}
\usepackage{amsmath}
\usepackage{amssymb}
\usepackage{amsthm}
\usepackage{mathtools}
\usepackage{paralist, tabularx}
\usepackage{enumitem}
\usepackage[utf8]{inputenc}
\usepackage{mathabx,epsfig}
\usepackage{tikz-cd}
\usepackage{verbatim}

\usepackage{etoolbox}

\usepackage{xcolor}
\definecolor{green}{RGB}{0,127,0}
\definecolor{redd}{RGB}{191,0,0}
\definecolor{red}{RGB}{105,89,205}
\usepackage[colorlinks=true]{hyperref}

\usepackage[notref, notcite]{showkeys}
\usepackage[cmtip,arrow]{xy}

\DeclareMathOperator{\Aut}{Aut}

\DeclareMathOperator{\Stab}{Stab}
\DeclareMathOperator{\st}{st}

\newcommand{\R}{{\mathbb{R}}}

\newcommand{\RR}{{\bar{R}}}

\newcommand{\RRaT}{(\RR,+)^{00}}

\newcommand{\RRiD}{\RR^{0}}
\newcommand{\RRiT}{\RR^{00}}

\newcommand{\C}{\mathfrak{C}}

\newcommand{\nref}[2]{\hyperref[#1]{\ref*{#1}$_{#2}$}}

\DeclareMathOperator{\SL}{{SL}}
\DeclareMathOperator{\tp}{{tp}}
\DeclareMathOperator{\cl}{{cl}}
\DeclareMathOperator{\dcl}{{dcl}}

\newtheorem{theorem}{Theorem}
\numberwithin{theorem}{section}
\newtheorem{lemma}[theorem]{Lemma}

\newtheorem{fact}[theorem]{Fact}
\newtheorem{proposition}[theorem]{Proposition}
\newtheorem{problem}[theorem]{Problem}

\newtheorem{question}[theorem]{Question}
\newtheorem{corollary}[theorem]{Corollary}

\newtheorem{clm}{Claim}
\newtheorem*{clm*}{Claim}

\theoremstyle{definition}
\newtheorem{definition}[theorem]{Definition}

\newtheorem{example}[theorem]{Example}

\theoremstyle{remark}
\newtheorem{remark}[theorem]{Remark}

\AtEndEnvironment{proof}{\setcounter{clm}{0}}
\newenvironment{clmproof}[1][\proofname]{\proof[#1]}{\endproof}

\newcommand{\orcidlogo}{\includegraphics[height=\fontcharht\font`\B]{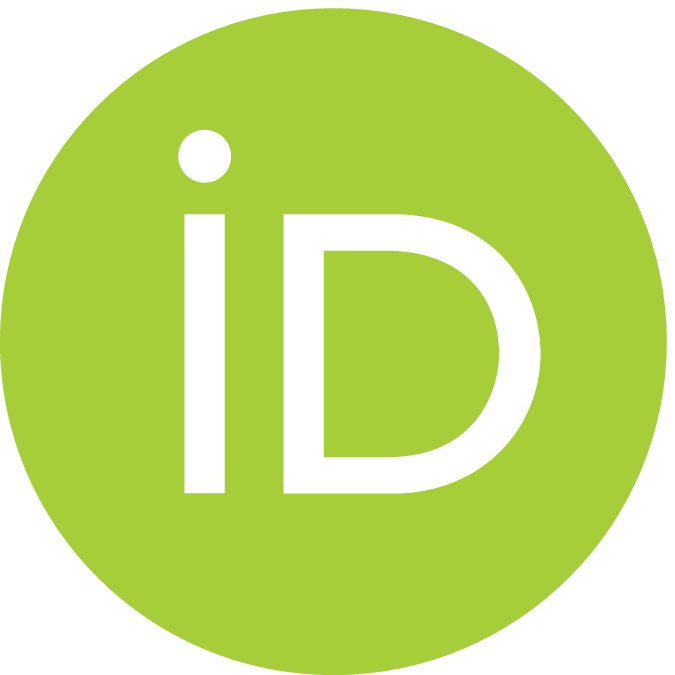}}
\newcommand{\orcid}[1]{\href{#1}{\orcidlogo #1}}

\usepackage[hyphenbreaks]{breakurl}

\usepackage{filecontents}

\usepackage[
backend=bibtex,
url=false,
isbn=false,
backref=true,
citestyle=alphabetic,
bibstyle=alphabetic,
autocite=inline,
maxnames=99,
minalphanames=4,
maxalphanames=4,
sorting=nyt,
]{biblatex}

\title{Locally compact models for approximate rings}

\author{Krzysztof Krupi\'{n}ski}
\thanks{\noindent The author is supported by the Narodowe Centrum Nauki grants no. 2016/22/E/ST1/00450 and 2018/31/B/ST1/00357.}
\address{\parbox{\linewidth}{Instytut Matematyczny, Uniwersytet Wroc{\l}awski\\
pl. Grunwaldzki 2, 50-384 Wroc{\l}aw, Poland}}
\address{\orcid{https://orcid.org/0000-0002-2243-4411}}
\email{Krzysztof.Krupinski@math.uni.wroc.pl}


\begin{document}
	
	\begin{abstract}
	By an approximate subring of a ring we mean an additively symmetric subset $X$ such that $X\cdot X \cup (X +X)$ is covered by finitely many additive translates of $X$.
We prove that each approximate subring $X$ of a ring has a locally compact model, i.e. a ring homomorphism $f \colon \langle X \rangle \to S$ for some locally compact ring $S$ such that $f[X]$ is relatively compact in $S$ and there is a neighborhood $U$ of $0$ in $S$ with $f^{-1}[U] \subseteq 4X + X \cdot 4X$ (where $4X:=X+X+X+X$). This $S$ is obtained as the quotient of the ring $\langle X \rangle$ interpreted in a sufficiently saturated model by its type-definable ring  connected component. The main point is to prove that this component always exists. In order to do that, we extend the basic theory of model-theoretic connected components of definable rings (developed in \cite{GJK} and \cite{KrRz}) to the case of rings generated by definable approximate subrings  and we answer a question from \cite{KrRz} in the more general context of approximate subrings. Namely, let $X$ be a definable (in a structure $M$) approximate subring of a ring and $R:=\langle X \rangle$. Let $\bar X$ be the interpretation of $X$ in a sufficiently saturated elementary extension and $\bar R := \langle \bar X \rangle$. It follows from \cite{MaWa} that there exists the smallest $M$-type-definable subgroup of $(\bar R,+)$ of bounded index, which is denoted by $(\bar R,+)^{00}_M$. We prove that  $(\bar R,+)^{00}_M + \bar R  \cdot (\bar R,+)^{00}_M$ is the smallest $M$-type-definable two-sided ideal of $\bar R$ of bounded index, which we denote by $\bar R^{00}_M$. Then $S$ in the first sentence of the abstract is just $\bar R/\bar R^{00}_M$ and $f: R \to \bar R/\bar R^{00}_M$ is the quotient map. In fact, $f$ is the universal ``definable'' (in a suitable sense)   locally compact model.

The existence of locally compact models can be seen as a general structural result about approximate subrings: every approximate subring $X$ can be recovered up to additive commensurability as the preimage by a locally compact model $f \colon \langle X \rangle \to S$ of any relatively compact neighborhood of $0$ in $S$. It should also have various applications to get more precise structural or even classification results. For example, in this paper, we deduce that every [definable] approximate subring $X$ of a ring of positive characteristic is additively commensurable with a [definable] subring contained in $4X + X \cdot 4X$. This easily implies that for any given $K,L \in \mathbb{N}$ there exists a constant $C(K,L)$ such that every $K$-approximate subring $X$ (i.e. $K$ additive translates of $X$ cover $X \cdot X \cup (X+X)$) of a ring of positive characteristic $\leq L$ is additively  $C(K,L)$-commensurable with a subring contained in $4X + X \cdot 4X$. Another application of the existence of locally compact models is a classification of finite approximate subrings of rings without zero divisors: for every $K \in \mathbb{N}$ there exists $N(K) \in \mathbb{N}$ such that for every  finite $K$-approximate subring $X$ of a ring without zero divisors either $|X| <N(K)$ or $4X + X \cdot 4X$ is a subring which is additively $K^{11}$-commensurable with $X$.

	\end{abstract}

	\maketitle

	\section{Introduction}

A subset $X$ of a group is called an {\em approximate subgroup} if it is symmetric (i.e. $e \in X$ and $X^{-1}=X$) and $X X \subseteq FX$ for some finite $F \subseteq \langle X \rangle$. Approximate subgroups were introduced by Tao in \cite{Tao} and have become one of the central objects in additive combinatorics.
A breakthrough in the study of the structure of approximate subgroups was obtained by Hrushovski in \cite{Hru}, where a locally compact model for any pseudofinite approximate subgroup (more generally, {\em near-subgroup}) $X$ was obtained by using model-theoretic tools, and in consequence also a Lie model  was found for some approximate subgroup commensurable with $X$. This paved the way for Breuillard, Green, and Tao to give a full classification of all finite approximate subgroups in \cite{BGT}.

Let $X$ be an approximate subgroup and $G := \langle X \rangle$. By a {\em locally compact [resp. Lie] model} of $X$ we mean a group homomorphism $f \colon \langle X \rangle \to H$ for some locally compact [resp. Lie] group $H$ such that $f[X]$ is relatively compact in $H$ and there is a neighborhood $U$ of the neutral element in $H$ with $f^{-1}[U] \subseteq X^m$ for some $m<\omega$.
(In this paper, locally compact spaces are Hausdorff by definition.)

By a {\em definable} (in some structure $M$)  {\em approximate subgroup} we mean an approximate subgroup $X$ of some group such that $X, X^2,X^3,\dots$ are all definable in $M$ and  $\cdot |_{X^n \times X^n} :X^n \times X^n \to X^{2n}$ is definable in $M$ as well. Naming the appropriate parameters, we will be assuming that the definable approximate subgroups are $0$-definable (i.e. without parameters).
If the approximate subgroup $X$ is definable in $M$, then in the definition of a locally compact model, one usually additionally requires {\em definability} of $f$ in the sense that for any open $U \subseteq H$ and compact $C \subseteq H$  such that $C \subseteq U$, there exists a definable (in $M$) subset $Y$ of $G$ such that $f^{-1}[C] \subseteq Y \subseteq f^{-1}[U]$. Note that in the abstract situation of an arbitrary approximate subgroup $X$, we can always equip the ambient group with the {\em full structure} (i.e. add all subsets of all finite Cartesian powers as predicates), and then $X$ becomes definable and the additional requirement of definability of locally compact models is automatically satisfied. In other words, definable approximate subgroups generalize abstract approximate subgroups.

It is folklore (see Corollary \ref{corollary: existence of locally compact models}) that for a definable (in a structure $M$) approximate subgroup $X$, the existence of a definable locally compact model is equivalent to the existence of some (equivalently, the smallest) $M$-type-definable subgroup of $\bar G:= \langle \bar X \rangle$ of bounded index, where $\bar X$ is the interpretation of $X$ in a monster model extending $M$ (see Subsection \ref{subsection: model theory}). This smallest subgroup is denoted by $\bar G^{00}_M$. By compactness, any type-definable subgroup of $\bar G$ is contained in some power $\bar X^m$; in particular, if $\bar G^{00}_M$ exists, it is necessarily contained in some $\bar X^m$.  The existence of $\bar G^{00}_M$ together with the requirement $\bar G^{00}_M \subseteq \bar X^m$ for a given $m$ is precisely equivalent to saying that there exists a sequence $(D_n)_{n < \omega}$ of definable, symmetric subsets of $X^m$ with the properties $D_{n+1}D_{n+1} \subseteq D_n$ and $D_n$ is generic (i.e. finitely many left translates of $D_n$ cover $X$) for all $n<\omega$. 
We have that if a definable locally compact model for $X$ exists (equivalently, $\bar G^{00}_M$ exists), then the quotient map $G \to \bar G/\bar G^{00}_M$ is the universal definable locally compact model (see Proposition \ref{proposition: locally compact model}).

Having a pseudofinite approximate subgroup $X$ of a group $M$, one can equip $M$ with a sufficiently rich structure (e.g. the full structure where all subsets of all finite Cartesian powers  are added as predicates). Let $G=\langle X \rangle$. Hrushovski proved  that for $\bar X$ being the interpretation of $X$ in the monster model extending $M$ and $\bar G:=\langle \bar X \rangle$, the component $\bar G^{00}_M$ exists and is contained in $\bar X^4$.  Then the quotient map $G \to \bar G/\bar G^{00}_M$ is the universal locally compact model for $X$. Next, using Yamabe's theorem, he deduced that there exists an approximate subgroup $Y$ commensurable with $X$ (i.e. finitely many left translates of $Y$ cover $X$ and vice versa) and contained in $X^4$ such that $Y$ has a Lie model. He proved a much more general result for the so-called near-subgroups (see \cite[Theorem 4.2]{Hru}). This was obtained as a consequence of a suitable ``stabilizer theorem'' in a stable context proved in \cite{Hru}. Some variants of Hrushovski's stabilizer theorem were established later in several papers by various authors. For example, Massicot and Wagner proved the existence of definable locally compact models for definably amenable approximate subgroups. More precisely, from \cite[Theorem 12]{MaWa} it follows that  if $X$ is a definable (in a structure $M$) definably amenable approximate subgroup, then in the monster model the group $\bar G :=\langle \bar X \rangle$ has the component $\bar G^{00}_M$ contained in $\bar X^4$ (see also Fact \ref{fact: main consequence of MaWa and Mas}). {\em Definable amenability} of $X$ means that there is an invariant under left translation, finitely additive measure $\mu$ on definable subsets of $\bar G: =\langle \bar X \rangle$ such that $\mu(\bar X)=1$. In particular, this applies in the case when $\bar G$ is abelian, as then $\bar G$ is an amenable group, and so $X$ is definably amenable (even amenable, e.g. see Lemma 6.1 in the first arXiv version of \cite{Hru2} or \cite[Proposition 5.8]{Ma}).

Wagner conjectured (see \cite[Conjecture 0.15]{Mas} and the paragraph after Theorem 1 in \cite{MaWa}) that a definable locally compact model always exists. This conjecture was refuted in \cite[Section 4]{HKP} (even in the abstract situation, where definability can be erased).  In \cite{Hru2}, Hrushovski proved the existence of locally compact and Lie models in a generalized sense involving quasi-homomorphisms, and used them to give complete classifications of approximate lattices in $\SL_n(\mathbb{R})$ and $\SL_n(\mathbb{Q}_p)$. This work is very advanced; among various tools, it uses a new locally compact group attached to a theory invented by Hrushovski as a counterpart of the Ellis group (or rather its canonical Hausdorff quotient) of  a first order theory which was defined and studied in \cite{KPR}, \cite{KNS}, and \cite{KrRz}.

The goal of the present paper is to see whether definable locally compact models for definable approximate rings always exist. In contrast to approximate groups, our main result yields a positive answer in full generality. We obtain it as an easy corollary of our main theorem which concerns some fundamental issues on model-theoretic connected components of approximate rings, answering in particular the main question from \cite{KrRz}, but in a more general context of approximate rings. Let us give some details.

In this paper, rings need not be commutative or unital.
There are various possible definitions of approximate subrings. One can define an approximate subring of a ring as an additively  symmetric subset $X$ such that $XX \cup (X+X) \subseteq (F \cup\{1\})X \cap (F+X)$ for some finite subset $F$ of the subring $\langle X \rangle$ generated by $X$. We will work with a more general definition saying that $XX \cup (X+X) \subseteq F+X$ for some finite $F \subseteq \langle X \rangle$; in  particular, $X$ is additively an approximate subgroup. Define recursively $X_n$, $n<\omega$, by: $X_0:=X$ and $X_{n+1}:= X_nX_n+(X_n +X_n)$. As we will see in Fact \ref{fact: by Lemma 5.5 of Bru}, if $X$ is an approximate subring, then each $X_n$ is covered by finitely many additive translates of $X$.
Important structural results (so called sum-product phenomena) on finite approximate subrings were obtained by Tao in \cite{Tao2}.

A key example of an approximate subring is an arbitrary compact neighborhood of $0$ in any locally compact ring, e.g. $X:= [-1,1]$ is an approximate subring of $\mathbb{R}$, and  $X:=\{ \sum_{i=-1}^\infty a_it^i: a_i \in \mathbb{F}_p\}$ is an approximate subring of the field of formal Laurent series $\mathbb{F}_p((t))$ over the finite field $\mathbb{F}_p$.

By a {\em definable} (in some structure $M$) {\em approximate subring} we mean an approximate subring $X$ such that $X_0,X_1,\dots$ are all definable in $M$ and $+$ and $\cdot$ restricted to any $X_n$ are also definable in $M$.

A [definable] locally compact model of a [definable] approximate subring $X$ is defined as a counterpart of a [definable] locally compact model of a [definable] approximate subgroup (see the paragraph preceding Proposition \ref{proposition: locally compact model}).  
As in the case of definable approximate subgroups, the existence of a definable locally compact model 
is equivalent to the existence of a suitable model-theoretic ring component of the ring $\bar R:= \langle \bar X \rangle$ generated by the interpretation $\bar X$ of $X$ in the monster model. Namely, we observe in Corollary \ref{corollary: existence of locally compact models} that a definable locally compact model for $X$ exists if and only if there exists some (equivalently, the smallest) $M$-type-definable two-sided ideal in $\bar R$ of bounded index. This smallest ideal is denoted by $\bar R^{00}_M$. By compactness, any type-definable subgroup of $(\bar R,+)$ is contained in some $\bar X_m$; in particular,  if $\bar R^{00}_M$ exists, it is contained in some $\bar X_m$.  Various model-theoretic connected components of definable rings were defined and studied in \cite{GJK} and \cite{KrRz}. In particular, it was shown in \cite{GJK} that in the definition of $\bar R^{00}_M$ ``two-sided ideal'' can be replaced by ``left ideal'' or ``right ideal'' or ``subring'' and in each case we get the same notion. The proofs also work for $\bar R= \langle \bar X \rangle$. By compactness, one easily shows that the existence of $\bar R^{00}_M$ together with the requirement $\bar R^{00}_M \subseteq \bar X_m$ for a given $m$ is precisely equivalent to saying that there exists a sequence $(D_n)_{n < \omega}$ of definable, additively symmetric subsets of $X_m$ with the properties $D_{n+1}D_{n+1} + (D_{n+1} + D_{n+1}) \subseteq D_n$ and $D_n$ is generic (i.e. finitely many additive translates of $D_n$ cover $X$) for all $n<\omega$. 

In \cite[Theorem 1.2]{KrRz}, it was shown that for a unital definable ring $\bar R$ we have $(\bar R,+)^{00}_M + \bar R  \cdot (\bar R,+)^{00}_M + \bar R  \cdot (\bar R,+)^{00}_M = \bar R^{00}_M$ (so we say that $(\bar R,+)^{00}_M$ generates an ideal in $2\frac{1}{2}$ steps), and for a definable ring of finite characteristic we have $(\bar R,+)^{00}_M+ \bar R  \cdot (\bar R,+)^{00}_M = \bar R^{00}_M$ (i.e.  $(\bar R,+)^{00}_M$ generates an ideal in $1\frac{1}{2}$ steps). It was left as a question (see \cite[Question 1.3]{KrRz}) if finitely many steps are enough for arbitrary definable rings (besides unital or finite characteristic rings a positive answer was also obtained for finitely generated rings, but with higher numbers of steps). It was also shown in Examples 8.1 and 8.2 of \cite{KrRz} that $1\frac{1}{2}$ steps is an optimal (i.e. cannot be decreased) bound on the number of steps needed to generate an ideal. 

In this paper, we prove that $1\frac{1}{2}$ steps is enough not only for arbitrary definable rings (answering \cite[Question 1.3]{KrRz}), but also for rings generated by definable approximate subrings, i.e. for $\bar R = \langle \bar X \rangle$ where $X$ is a definable approximate subring. Namely, in Theorem \ref{theorem: main theorem}, we show that  $(\bar R,+)^{00}_M+ \bar R  \cdot (\bar R,+)^{00}_M = \bar R^{00}_M$; in particular, $\RR^{00}_M$ exists. 
From this, we deduce in  Corollary \ref{corollary: locally compact model exists} that a definable locally compact model exists for an arbitrary definable approximate subring, and the quotient map $R \to \bar R/\bar R^{00}_M$ is the universal such model.

In fact, we show that for an arbitrary small subset $A$ of the monster model $(\bar R,+)^{00}_A+ \bar R  \cdot (\bar R,+)^{00}_A$ is an invariant over $A$ two-sided ideal of bounded index (actually the smallest one, denoted by $\bar R^{000}_A$). For a definable $\bar R$ it gives us $(\bar R,+)^{00}_A+ \bar R  \cdot (\bar R,+)^{00}_A = \bar R^{00}_A.$ In our general context of  $\bar R = \langle \bar X \rangle$ (where $X$ is a 0-definable approximate subring), we get the last equality assuming that $R \subseteq \dcl(A)$ (so e.g. for $A=R$ or $A=M$).

The existence of locally compact models for arbitrary approximate subrings can be seen as a general structural result about all approximate subrings. Namely, every approximate subring $X$ is additively commensurable with the preimage by a locally compact model $f \colon \langle X \rangle \to S$ of any relatively compact neighborhood of $0$ in $S$ (where two subsets of a ring are said to be {\em additively commensuarble} if they are commensurable as subsets of the additive group); see Fact \ref{fact: preimage of open is generic}(2). After replacing $S$ by the closure of the image of $f$, we conclude that for some approximate subring $Y \supseteq \ker(f)$ additively commensurable with $X$, $Y/\ker(f)$ becomes a dense subset of an open, relatively compact neighborhood of $0$ in $S$. As an example, one can consider the approximate subring $X:=\mathbb{Q} \cap (-1,1)$ of the field of rationals. Then the identity embedding of $\langle X \rangle=\mathbb{Q}$ to $\mathbb{R}$ is a locally compact model, and $X$ is a dense subset of $(-1,1)$.

The existence of locally compact models also leads to more precise structural or even classification results about approximate subrings.

\begin{problem}
Classify all finite (more generally also infinite) approximate subrings.
\end{problem}

In Theorem \ref{theorem: application of the main theorem}, we deduce from our main result on the existence of locally compact models and a basic structural fact on locally compact rings that every definable approximate subring $X$ of a ring of positive characteristic is additively commensurable with a definable subring contained in $4X + X \cdot 4X$. Using ultraproducts, this implies Corollary \ref{corollary: application to finite approximate subrings} which says that for any given $K,L \in \mathbb{N}$ there exists a constant $C(K,L)$ such that every $K$-approximate subring $X$ of a ring of positive characteristic $\leq L$ is additively $C(K,L)$-commensurable with a subring contained in $4X + X \cdot 4X$ (where $C(K,L)$-commensurability means that at most $C(K,L)$ translates are enough).  

In Theorem \ref{theorem: classification without zero divisors}, the existence of locally compact models and ultraproducts are used to classify finite approximate subrings of rings without zero divisors: for every $K \in \mathbb{N}$ there exists $N(K) \in \mathbb{N}$ such that for every  finite $K$-approximate subring $X$ of a ring without zero divisors either $|X| <N(K)$ or $4X + X \cdot 4X$ is a subring which is additively $K^{11}$-commensurable with $X$. This is a variant of the classification of finite approximate fields from \cite[Theorem 5.3]{Bre}. A version of Thm. \ref{theorem: classification without zero divisors} with a less precise conclusion follows from \cite[Theorem 1.4]{Tao2}. 

In the zero characteristic case, it is not true that every approximate subring is additively commensurable with a subring, and the zero characteristic case with zero divisors remains open regarding proving structural results via applications of locally compact models. 
For example, one could try to extend the context of Corollary  \ref{corollary: application to finite approximate subrings} to some other families of finite approximate subrings, or, in all characteristics, one could try to prove (or reprove) some versions of the sum-product phenomena from \cite{Tao2}. A general idea to prove an asymptotic structural result on finite approximate subrings would be to consider an ultraproduct of a sequence of rings generated by finite approximate subrings yielding  a counter-example, pass to the universal locally compact model of the ultraproduct of the finite approximate subrings in question constructed in this paper, use suitable structural theorems on locally compact rings, come back to the to the ultraproduct, and try to get a contradiction. This worked easily to get Corollary \ref{corollary: application to finite approximate subrings} (even for infinite approximate subrings). This kind of idea was also used in \cite{BGT} for finite approximate subgroups using Lie models, but it required very hard technical work to get the desired classification result for finite approximate subgroups.

\section{Preliminaries}\label{section: preliminaries}

\subsection{Approximate rings}\label{subsection: approximate rings}

[For $K \in \mathbb{N}$] by a {\em [$K$-]approximate subring} of a ring we mean an additively symmetric subset $X$ of this ring such that $X\cdot X \cup (X +X) \subseteq F+X$ for some finite [resp. of size $K$] subset $F$ of the ring generated by $X$, which will be denoted by $\langle X \rangle$. Then $X$ is clearly additively an approximate subgroup. The sequence $(X_n)_{n<\omega}$ is defined recursively: $X_0:=X$ and $X_{n+1}:= X_nX_n+(X_n +X_n)$. For $m \in \mathbb{N}$ let $X^m$ [resp. $X^{\leq m}$] denote  the set of products of $m$ [resp. at most $m$] elements of $X$, and $m(X^{\leq m})$ the set of sums of $m$ elements which are products of at most $m$ elements of $X$. Adapting the proof of \cite[Lemma 5.5]{Bre}, we get the following fact.

\begin{fact}\label{fact: by Lemma 5.5 of Bru}
If $X$ is an approximate subring, then for every $m \in \mathbb{N}_{>0}$, $m(X^{\leq m})$ is covered by finitely many additive translates of $X$. In particular, every $X_n$ is covered by finitely many additive translates of $X$.
\end{fact}

\begin{proof}
The second part follows directly from the first, as $X_n$ is contained in $m(X^{\leq m})$ for a sufficiently large $m$. 

Let $F \subseteq \langle X \rangle$ be such that $XX \cup (X+X) \subseteq F+X$.

\begin{clm}
For every $x \in \langle X \rangle$, $xX$ is covered by finitely many additive translates of $X$.
\end{clm}

\begin{clmproof}
First, by induction on $m$, we show that for any $x_0,\dots,x_{m-1} \in X$ one has that $x_{m-1}\dots x_0 X$ is covered by finitely many additive translates of $X$. For $m=1$ we have $x_0X \subseteq XX \subseteq F +X$. For the induction step, assume that $x_{m-1}\dots x_0X \subseteq G +X$ for some finite $G$. Then $x_m\dots x_0X \subseteq x_m(G+X) =x_mG +XX \subseteq x_mG+F +X$ and $x_mG +F$ is clearly finite.

Next, by induction on $m$, we show that whenever $x_0,\dots,x_{m-1} \in \langle X \rangle$ are such that for every $i<m$, $x_iX \subseteq F_i + X$ for some finite $F_i$, then $(x_{m-1} + \dots +x_0)X$ is also covered by finitely many additive translates of $X$. For the induction step, assume that $(x_{m-1} + \dots +x_0)X \subseteq G +X$ for some finite $G$. Then $(x_m + \dots +x_0)X \subseteq F_m+X + G +X \subseteq F_m+G +F +X$ and $F_m+G +F$ is clearly finite.
\end{clmproof}

\begin{clm}
$X^m \subseteq F_m +X$ for some finite $F_m$.
\end{clm}

\begin{clmproof}
We prove it by induction on $m$. For $m=1$ it is trivial. For the induction step, assume that $X^m \subseteq F_m +X$ for some finite $F_m$. By Claim 1, for every $x \in F_m$ there exists a finite $F_x$ such that $xX \subseteq F_x +X$. Hence, $F_mX \subseteq X + \bigcup_{x \in F_m} F_x$ and $G_m: = \bigcup_{x \in F_m} F_x$ is finite. So $X^{m+1} =X^mX \subseteq F_mX +XX \subseteq G_m+X +F +X \subseteq G_m+F +F + X$ and $G_m+F +F$ is finite. 
\end{clmproof}

Since $X^{ \leq m} = X^1 \cup \dots \cup X^{m}$, by Claim 2, we get $X^{\leq m} \subseteq F_m' +X$ for some finite $F_m'$. By an easy induction, we conclude that $m(X^{ \leq m})$ is covered by finitely many additive translates of $X$.
\end{proof}

Two subsets $X$ and $Y$ of a ring $R$ are said to be {\em additively [$K$-]commensurable} if each of these subsets is covered by finitely many [at most $K$] additive translates of the other one.


\begin{corollary}
If $X$ is an approximate subring, then any additively symmetric subset $Y$ of $\langle X \rangle$ additively commensurable with $X$ is also an approximate subring.
\end{corollary}

\begin{proof}
By assumption, $Y \subseteq \bigcup_{i=0}^{m-1} x_i + X$ for some $x_i \in \langle X \rangle$. Since all $x_i \in X_n$ for some $n$, we get that $Y \subseteq X_{n+1}$. Thus, $Y^2 + (Y+Y)\subseteq X_{n+2}$. Hence,  by Fact \ref{fact: by Lemma 5.5 of Bru}, $Y^2 +(Y+Y)$ is covered by finitely many additive  translates of $X$, and so, by assumption, it is covered by finitely many additive translates of $Y$.
\end{proof}

However, it is not  true that if an additively  symmetric subset $Y$ of a ring $R$ is additively  commensurable with an approximate subring (or even actual subring) of $R$, then $Y$ is an approximate subring. As an example, take a proper field extension $K \subseteq L$, where $K$ is an infinite field, and pick $t \in L \setminus K$. Then $Y:=( t+K) \cup \{0\} \cup (-t+K)$ is clearly additively symmetric and commensurable with $K$, while it is easy to see that it is not an approximate subring of $L$. Taking $t$ transcendental over $K$, we even have that for every $m$, $Y_m$ is not an approximate subring.

\subsection{Model theory}\label{subsection: model theory}

Let $T$ be a complete first order theory in a language $\mathcal{L}$. For a model $M$ of $T$ and $A \subseteq M$, a {\em type} over $A$ is a consistent collection of formulas with parameters from $A$. A {\em monster model} of $T$ (often denoted by $\mathfrak{C}$) is a $\kappa$-saturated and strongly $\kappa$-homogeneous model of $T$ for a sufficiently large cardinal $\kappa$; usually it suffices to assume that $\kappa$ is a strong limit cardinal greater that $|\mathcal{L}|$ (i.e. the cardinality of the set all formulas in $\mathcal{L}$). {\em $\kappa$-saturation} means that every type over a set of parameters from $\mathfrak{C}$ of cardinality less than $\kappa$ has a realization in $\mathfrak{C}$; {\em strong $\kappa$-homogeneity} means that every elementary map between subsets of $\mathfrak{C}$ of cardinality smaller than $\kappa$ extends to an automorphism of $\mathfrak{C}$. It is a common thing in model theory to work in a fixed monster model, which always exists (by using model-theoretic compactness and a suitable recursive construction). A subset of $\mathfrak{C}$ is said to be {\em small} if its cardinality is smaller than $\kappa$; a cardinal is {\em bounded} if it is smaller than $\kappa$. It is very convenient to work in a monster model, especially when one deals with definable approximate groups or rings and with the model-theoretic connected components of the groups or rings generated by them.

Working in a model $M$ of $T$, for $A\subseteq M$, an {\em $A$-definable} set is the set of realizations in $M$ of a formula with parameters from $A$; a {\em definable set} is an $M$-definable set; instead of ``$\emptyset$-definable'' we will write ``0-definable''. Working in $\C$, for a small $A \subseteq \C$, an {\em $A$-type-definable} set is the set of realizations in $\C$ of a type over $A$; a {\em type-definable set} is an $A$-type-definable set for some small $A \subseteq \C$. (The empty set is also considered as type-definable if needed.) Finally, a subset of $\C$ (or of a Cartesian power of $\C$) is said to be {\em $A$-invariant} if it is invariant under $\Aut(\C/A)$ ($=$ the pointwise stabilizer of $A$); in contrast to definability and type-definability, {\em invariance} means 0-invariance (i.e.  invariance under $\Aut(\C)$). Throughout the paper, $\C$ is always chosen as a monster model with respect to $M$, that is $\C \succ M$ and the degree of saturation (i.e. $\kappa$ above) of $\C$ is bigger than $|M|$.

A group [ring] $G$ is said to be {\em $A$-definable} if both the universe $G$ and the group operation [resp. $\cdot$ and $+$] are $A$-definable. {\em Type-definable} and {\em invariant} groups [rings] are defined analogously (working in $\C$). 

Definable (in $M$) approximate subgroups and subrings were defined in the introduction. Adding finitely many parameters from $M$ to the language, we can and do assume that they are 0-definable. More general notions are those of $\bigvee$-definable (or ind-definable) groups and rings, but we will not go into that in this paper.

For a definable subset $D$ of $M$, by $\bar D$ we will usually denote its interpretation in $\C$, 
but with one exception. If $X$ is a definable (in $M$) approximate subgroup [or subring] and $R:=\langle X \rangle$, then $\bar X$ is the interpretation of $X$ in $\C$, but $\bar R$ will stand for $\langle \bar X \rangle$. It may happen that $R$ is definable in $M$, and if it is not the case that $R =X^m$ [or $R=X_m$ in the case of rings] for some $m$, then $\bar R$ is not the interpretation of the definable $R$ in $\C$. This is because, by saturation of $\C$, the fact that $\bar R$ is definable is equivalent to $\bar R=\bar X^m$  [or $\bar R=\bar X_m$ in the case of rings] for some $m$. 

For any $a \in \C$ and $A \subseteq \C$, by $\tp(a/A)$ we denote the {\em type of $a$ over $A$}, that is the collection of all formulas over $A$ realized by $a$. By $\dcl(A)$ we denote the {\em definable closure} of $A$, i.e. the collection of all elements which are fixed by $\Aut(\C/A)$.

When $D$ is a definable set in $M$ and $C$ is a compact space, then a function $f \colon D \to C$ is said to be {\em definable} if the preimages of any two disjoint closed subsets of $C$ can be separated by a definable subset of $D$. (This is essentially saying that $f$ is a continuous logic formula, but we will not use any continuous logic terminology in this paper.) By \cite[Lemma 3.2]{GPP}, this is equivalent to saying that $f$ extends to an {\em $M$-definable map} $\bar f \colon \bar D \to C$ in the sense that the preimage of any closed subset of $C$ is $M$-type-definable. Such an extension $\bar f$  is unique and given by $\bar f(a)=\bigcap_{\varphi(x) \in \tp(a/M)} \cl (f[\varphi(M) \cap D])$ (where $\cl$ denotes the closure in $C$).
In Section \ref{section: components of approximate rings}, we extend these considerations to definable approximate subgroups and subrings.

\subsection{Model-theoretic connected components of definable groups and rings}

We recall below some facts on model-theoretic connected components of definable groups and rings. While definable groups have played an important role in model theory for many years, the components of rings were introduced recently in \cite{GJK} where they were used to compute Bohr compactifications of some groups of matrices, e.g. both the discrete and continuous Heisenberg group. They were further studied in \cite{KrRz}.
	
Let $R$ be a 0-definable group [resp. ring], $\RR=R(\C)$, and $A \subseteq \C$ be a small set of parameters.

\begin{itemize}
		\item ${\bar R}^0_A$ is the intersection of all $A$-definable, finite index subgroups [ideals] of $\bar R$.
		\item ${\bar R}^{00}_A$ is the smallest $A$-type-definable, bounded index subgroup [ideal] of $\bar R$.
		\item ${\bar R}^{000}_A$ is the smallest $A$-invariant, bounded index subgroup [ideal] of $\bar R$.
	\end{itemize}
	
	We did not specify whether the ideals above are left, right, or two-sided. This is because of Proposition 3.6, Corollary 3.7, and Proposition 3.10 from \cite{GJK} which tell us that
	
	\begin{fact}\label{fact: ideal component}
		The above components of the ring $\bar R$ do not depend on the choice of the version (left, right, or two-sided) of the ideals. Moreover, instead of ``ideal'' we can equivalently write ``subring'' in the above definitions.
	\end{fact}

In the case of a definable group $R$, it is easy to see (cf.\ for example \cite[Lemma 2.2(3)]{Gis1}) that $\bar R^0_A, \bar R^{00}_A, \bar R^{000}_A$ are always normal subgroups of $\RR$.

For a definable group [or ring] $R$ and a small $A \subseteq \C$, ${\bar R}^{000}_A \leq {\bar R}^{00}_A \leq {\bar R}^0_A$. It is easy to see (cf. \cite[Lemma 2.2(1)]{Gis1}) that all these components exist and their indices in $\bar R$ are in fact bounded by $2^{|\mathcal{L}|+|A|}$. 

If $S$ is a type-definable, normal subgroup [two-sided ideal] in $\RR$ of bounded index, then $\RR/S$ is equipped with the \emph{logic topology}: closed sets are those whose preimages under the quotient map are type-definable. This makes the quotient $\RR/S$ a compact (topological) group [ring] (for the case of groups see \cite[Section 2]{Pil}; for rings it remains to check that multiplication is continuous which is an easy exercise).

A {\em compactification} of a (discrete) group [resp. ring] $R$ is a homomorphism $f \colon R \to C$ with dense image, where $C$ is a compact group [ring]. A {\em definable compactification} of $R$ is a compactification which is a definable map as defined in Subsection \ref{subsection: model theory}. The {\em Bohr compactification} of $R$ is the unique (up to isomorphism) universal compactification $h \colon R \to U$ of $R$ (universality means that for any other compactification $f \colon R \to C$ there exists a unique continuous homomorphism $g: U \to C$ such that $f=g \circ h$); and similarly in the definable version.

By \cite[Proposition 3.4]{GPP} and \cite[Proposition 3.28]{GJK}, we know that the quotient map $R \to \RR/\RR^{00}_M$ is the definable Bohr compactification of the group [ring] $R$. The idea of the proof is very simple. If $f \colon G \to C$ is a definable compactification, one extends it uniquely to an $M$-definable map $\bar f \colon \bar R \to C$ and checks that $\bar f$ is also a homomorphism which factors through the quotient map $\bar R \to \bar R/\bar R^{00}_M$. Similarly,   the quotient map $R \to \RR/\RR^{0}_M$ is the universal definable profinite compactification of $R$. So for example the equality $\RRiT_M=\RRiD_M$ means precisely that both compactifications coincide. When $R$ is equipped with the full structure, we can
erase the adjective ``definable'' and we get classical notions of compactification (described in a model-theoretic way).

Using the classical fact that compact unital or finite characteristic rings are profinite, we get that whenever a 0-definable ring $R$ is unital or of finite characteristic, then $\RR^{0}_A=\RR^{00}_A$ (see \cite[Corollary 2.10]{KrRz}). The following is \cite[Theorem 1.2]{KrRz}:

\begin{fact}\label{theorem: Main Theorem}
		Let $R$ be a 0-definable ring and $A \subseteq \C$ a small set of parameters.
		\begin{enumerate}
			\item If $R$ is unital, then $\RRaT_A + \RR \cdot \RRaT_A + \RR \cdot \RRaT_A = \RR^{000}_A=\RR^{00}_A =\RR^{0}_A$.
			\item If $R$ is of positive characteristic (not necessarily unital), then $\RRaT_A + \RR \cdot \RRaT_A = \RR^{000}_A=\RR^{00}_A =\RR^{0}_A$.
		\end{enumerate}
	\end{fact}

It was asked in \cite{KrRz} whether a similar fact holds for arbitrary 0-definable $R$ (except ``$=\RR^{0}_A$'', which fails in general; e.g. in some rings with zero multiplication) and if yes, how many steps are needed. In Section \ref{section: main results}, we will answer this question by proving that for every 0-definable ring $R$,  $\RRaT_A + \RR \cdot \RRaT_A = \RR^{000}_A=\RR^{00}_A$, so $1\frac{1}{2}$ steps always suffice. On the other hand, Examples 8.1 and 8.2 of \cite{KrRz} show that one cannot decrease the number of steps to $1$ (i.e. $(\RR \cup \{1\}) \cdot \RRaT_A$ need not be an additive subgroup), even for commutative, unital rings of finite characteristic.

\section{Model-theoretic connected components of definable approximate groups and rings}\label{section: components of approximate rings}

For a definable (in some $M$) approximate subgroup [subring] $X$, $R:=\langle X \rangle$, $\bar R = \langle \bar X \rangle$, and a small set of parameters $A \subseteq \C$, we define the following components.
\begin{itemize}
		\item ${\bar R}^{00}_A$ is the smallest $A$-type-definable, bounded index subgroup [two-sided ideal] of $\bar R$.
		\item ${\bar R}^{000}_A$ is the smallest $A$-invariant, bounded index subgroup [two-sided ideal] of $\bar R$.
	\end{itemize}

In contrast to definable groups, the existence of ${\bar R}^{00}_A$ for definable approximate subgroups [subrings] is a non-trivial issue. 

Both above components for definable approximate subgroups were studied in Section 4 of \cite{HKP}. In particular, Proposition 4.3 of \cite{HKP} yields the existence and a description of ${\bar R}^{000}_A$ which implies that $[\RR:{\bar R}^{000}_A] \leq 2^{|\mathcal{L}| +|A|}$. On the other hand, \cite[Subsection 4.3]{HKP} yields an example where ${\bar R}^{00}_A$ does not exist. The existence of ${\bar R}^{00}_A$ is equivalent to the existence of some $A$-type-definable subgroup of $\bar R$ of bounded index (as then the intersection of all such subgroups  is $A$-type-definable of index  $\leq  2^{|\mathcal{L}| +|A|}$, so equals ${\bar R}^{00}_A$). 

In the context of definable approximate subrings, we will prove in Section \ref{section: main results} that  $(\bar R,+)^{00}_A + \bar R  \cdot (\bar R,+)^{00}_A=\RR^{000}_A$ which further equals $\RR^{00}_A$ provided that $R \subseteq \dcl(A)$; in particular, in contrast to definable approximate subgroups, ${\bar R}^{00}_A$ always exists for definable approximate subrings (under the assumption that $R \subseteq \dcl(A)$).  For completeness notice that the existence of  ${\bar R}^{000}_A$ 
is clear: the intersection of all $A$-invariant, bounded index, two-sided ideals of $\bar R$ will be $A$-invariant and of bounded index $\leq [\RR: (\RR,+)^{000}_A] \leq  2^{|\mathcal{L}| +|A|}$.

The proofs of statements 3.3, 3.5, and 3.6(i,ii) of \cite{GJK} go through with very minor adjustments to conclude with

\begin{proposition}\label{proposition: side is irrelevant}
		The components $\RR^{00}_A$ and $\RR^{000}_A$ of the ring $\bar R$ do not depend on the choice of the version (left, right, or two-sided) of the ideals. Moreover, instead of ``two-sided ideal'' we can equivalently write ``subring'' in the above definitions.
\end{proposition}

Let $\RR$ be as in the first sentence of this section.
If $I$ is a type-definable, normal subgroup [two-sided ideal] in $\RR$ of bounded index, then $\RR/I$ is equipped with the \emph{logic topology}: a subset $F \subseteq \RR/I$ is closed if the sets $\pi^{-1}[F] \cap \bar X^m$ [resp. $\pi^{-1}[F] \cap \bar X_m$] are type-definable for every $m \in \omega$, where $\pi \colon \RR \to \RR/I$ is the quotient map. This makes the quotient $\RR/I$ a locally compact topological group [resp. ring]. For groups it appeared first time in Section 7 of \cite{HPP}, and then stood behind the model-theoretic approach to approximate subgroups. For rings one additionally has to check that multiplication is continuous, which is an easy exercise. Let us only remark that by compactness (or rather saturation of $\C$), $I \subseteq \bar X^m$ [resp. $I\subseteq \bar X_m$] for some $m \in \omega$. 
A compact neighborhood of $e$ [resp. of $0$] is for example the set $\bar X^n/I$ [resp. $\bar X_n/I$] for any $n \geq m$, with an open neighborhood of $e$ [resp. $0$] contained in it being $\{a/I: aI \subseteq \bar X^n\}$ [resp. $\{a/I: a+I \subseteq \bar X_n\}$]. The compact subsets of $\RR/I$ are those with type-definable preimage under $\pi$ (and so contained in some $\bar X^n$ [resp. $\bar X_n$]).

As was already explained in the introduction, one can extend the notion of a definable map from a definable set to a compact space to homomorphisms from groups [rings] generated by definable approximate subgroups [subrings] to locally compact groups. Namely, for a definable approximate subgroup [subring] $X$ and a locally compact group [resp. ring] $H$, a homomorphism $f \colon R \to H$ such that $f[X]$ is relatively compact in $H$ will be called {\em  definable} if for any open $U \subseteq H$ and compact $C \subseteq H$  such that $C \subseteq U$, there exists a definable (in $M$) subset $Y$ of $R$ such that $f^{-1}[C] \subseteq Y \subseteq f^{-1}[U]$. 

\begin{lemma}\label{lemma: extension of definable map}
Let $X$ be a definable approximate subgroup [subring] and $H$ a locally compact group [ring]. Let $f \colon R \to H$ be a homomorphism such that $f[X]$ is relatively compact in $H$.
\begin{enumerate}
\item If $f$ is definable, then it extends uniquely to a map $\bar f \colon \RR \to H$ such that $\bar f^{-1}[C] \cap \bar X^m$ [resp. $\bar f^{-1}[C] \cap \bar X_m$] is $M$-type-definable for every $m$ and for every closed $C \subseteq H$. This unique $\bar f$ is a homomorphism.
\item Assume additionally that there is a neighborhood $V$ of $e$ [resp. of $0$] in $H$ such that $f^{-1}[V] \subseteq X^m$ [resp. $f^{-1}[V] \subseteq X_m$] for some $m$. Then, if $f$ extends to  some $\bar f \colon \RR \to H$ as in (1), then $f$ is definable.
\end{enumerate}
\end{lemma}

\begin{proof}
Let us focus on the case of an approximate subgroup $X$; the case of an approximate subring is completely analogous (working with $X_m$ in place of $X^m$).

(1) Let $H_m:=\cl(f[X^m])$. Since $H_1 =\cl(f[X])$ is compact by assumption and $H_m=\cl(f[X]^m)$, we get that $H_m=H_1^m$ is also compact. Therefore, by the assumption of (1), $f |_{X^m} \colon X^m \to H_m$ is a definable map from a definable set to a compact space. So it extends uniquely to an $M$-definable function $\bar f_m \colon \bar X^m \to H_m$, as explained in the last paragraph of Subsection \ref{subsection: model theory}. By the explicit formulas for the $\bar f_m$'s, we see that $\bar f_1 \subseteq \bar f_2 \subseteq \dots$. 
So $\bar f:= \bigcup_m \bar f_m$ is the desired extension of $f$. Its uniqueness follows from the uniqueness of the $\bar f_m$'s after noticing that for any other $\bar f' \colon \RR \to H$ as in (1) we have $\bar f'[\bar X^m] \subseteq H_m$ (which holds as $\bar f'^{-1}[H_m] \cap \bar X^m$ is an $M$-type-definable set containing $X^m$, and so $\bar f'^{-1}[H_m] \cap \bar X^m=\Bar X^m$).

To see that $\bar f$ is a homomorphism, one can apply the argument from \cite[Proposition 3.4]{GPP}. Namely, since $\bar f_m \subseteq \bar f$ and any $a \in  \RR= \langle\bar X \rangle$ belongs to some $\bar X^m$,  we have $\bar f(a)=\bigcap_{\varphi(x) \in \tp(a/M)} \cl (f[\varphi(M) \cap X^m])$. Consider any $a,b \in \RR$. Choose $m$ such that $a,b,ab \in \bar X^m$. Let $p =\tp(a/M)$, $q = \tp(b/M)$, and $r=\tp(ab/M)$. Then 
\begin{eqnarray*}
\{\bar f (ab)\} = \bigcap_{\varphi(x) \in r} \cl (f[\varphi(M) \cap X^m])    \subseteq   \bigcap_{\varphi(x) \in p, \psi(x) \in q} \cl (f[\varphi(M)\cdot \psi(M) \cap X^m])   =   \\
\bigcap_{\varphi(x) \in p, \psi(x) \in q} \cl (f[\varphi(M) \cap X^m] \cdot f[\psi(M) \cap X^m])  = \\
 \bigcap_{\varphi(x) \in p, \psi(x) \in q} \cl (f[\varphi(M) \cap X^m]) \cdot \cl(f[\psi(M) \cap X^m]) =\\
 \bigcap_{\varphi(x) \in p} \cl (f[\varphi(M) \cap X^m])  \cdot \bigcap_{\psi(x) \in q} \cl (f[\psi(M) \cap X^m]) = \{\bar f(a) \bar f(b)\},
\end{eqnarray*}
where the third and fourth equality uses compactness of $H_m$ and  continuity of $\cdot$ (the fourth equality also uses the fact that the families $\{ \varphi(M) \cap X^m: \varphi(x) \in p\}$ and $\{\psi(M) \cap X^m: \psi(x) \in q\}$ are closed under finite intersections).

(2) Consider a compact $C \subseteq H$ and an open $U \subseteq H$ such that $C \subseteq U$. 
Choose a neighborhood $W$ of $e$ such that $W^{-1}W \subseteq V$ (where $V$ is from the assumption in (2)). Since $C$ is compact, $C \subseteq \bigcup_{i<n} g_iW$ for some $n<\omega$ and $g_i \in H$. Hence, $f^{-1}[C] \subseteq \bigcup_{i<n} f^{-1}[g_iW]$. Pick $a_i \in f^{-1}[g_iW]$ (if there is any) for $i<n$. Then  $f^{-1}[g_iW] \subseteq a_i f^{-1}[W^{-1}W] \subseteq a_if^{-1}[V] \subseteq a_iX^m$. So $f^{-1}[C] \subseteq X^k$ for some $k$. On the other hand, by the property of $\bar f$, we have that $\bar f^{-1}[C] \cap \bar X^k$ and $\bar f^{-1}[H \setminus U] \cap \bar X^k$ are disjoint $M$-type-definable sets. So they can be separated by $\bar Y$ for some  ($M$-)definable subset $Y$ of $X^k$. Hence, $f^{-1}[C] \subseteq Y \subseteq f^{-1}[U]$. 
\end{proof}

By a {\em definable locally compact model} of $R$ we mean a definable homomorphism $f \colon R \to S$ for some locally compact group [resp. ring] $S$ such that $f[X]$ is relatively compact in $S$ and there is a neighborhood $U$ of $e$ [resp. $0$] with $f^{-1}[U] \subseteq X^m$ [resp. $f^{-1}[U] \subseteq X_m$] for some $m<\omega$. It is well-known (at least for approximate groups) that for $A \subseteq M$, the quotient map $R \to \RR/\RR^{00}_A$ is a definable locally compact model of $R$. We give a proof below for the readers convenience, and we additionally prove universality of this model for $A:=M$.

\begin{proposition}\label{proposition: locally compact model}
Let $A \subseteq M$ and assume that $\RR^{00}_A$ exists. The quotient map $h \colon R \to \RR/\RR^{00}_A$ is a definable locally compact model, which for $A:=M$ is universal in the sense that for any other definable locally compact model $f \colon R \to S$ there is a unique continuous homomorphism $g \colon \RR/\RR^{00}_M \to S$ such that $f=g \circ h$. 
\end{proposition}

\begin{proof}
We skip the proof that $\RR/\RR^{00}_A$ is a locally compact group [ring] (see the discussion after Proposition \ref{proposition: side is irrelevant}). Since $h[X] \subseteq \bar X/\RR^{00}_A$ and the last set is compact, we get that $h[X]$ is relatively compact. As remarked above, if we choose $m$ with $\RR^{00}_A \subseteq \bar X^m$ [resp.  $\RR^{00}_A \subseteq \bar X_m$], then $U:=\{a/\RR^{00}_A: a\RR^{00}_A \subseteq \bar X^m\}$ [resp. $U:=\{a/\RR^{00}_A: a+\RR^{00}_A \subseteq \bar X_m\}$] is an open neighborhood of $e$ [resp. $0$]; and clearly $h^{-1}[U] \subseteq X^m$ [resp.  $h^{-1}[U] \subseteq X^m$].   To show definability of $h$, consider any compact $C \subseteq \RR/\RR^{00}_A$ and open $V \subseteq \RR/\RR^{00}_A$ such that $C \subseteq V$. Let $\bar h \colon \RR \to \RR/\RR^{00}_A$ be the quotient map. Then $\bar h^{-1}[C]$ is a type-definable subset of some $\bar X^n$ which is disjoint from the type-definable set $\bar h^{-1}[(\RR/\RR^{00}_A) \setminus V] \cap \bar X^n$. It remains to show that these disjoint sets are $M$-invariant, as then they are $M$-type-definable, so, being disjoint, they can be separated by $\bar Y$ for some definable subset $Y$ of $X^n$; then clearly $h^{-1}[C] \subseteq Y \subseteq h^{-1}[V]$, as required. The fact that these type-definable sets are $M$-invariant follows from the fact that the relation of having the same type over $M$ is the finest $M$-invariant, bounded  equivalence relation on  $\bar R$ and so it refines the relation of lying in the same coset of $\RR^{00}_A$ (as $A \subseteq M$).

Observe that $R/\RR^{00}_A$ is dense in $\RR/\RR^{00}_A$. Indeed, take a non-empty open $V \subseteq \RR/\RR^{00}_A$. Then the preimage $\bar h^{-1}[V]$ is a union of definable sets; and, as above we can find these sets to be definable over $M$. Since this union is non-empty, at least one of these $M$-definable sets is non-empty, and so it intersects $R$, which shows that $V \cap (R/\RR^{00}_A)$ is non-empty. By the density of $R/\RR^{00}_A$, uniqueness in the universal property becomes clear.

For the existence, consider any definable locally compact model $f \colon R \to S$. By Lemma \ref{lemma: extension of definable map}, $f$ extends to a homomorphism $\bar f \colon \RR \to S$ such that $\bar f^{-1}[C] \cap \bar X^k$ [resp. $\bar f^{-1}[C] \cap \bar X_k$] is $M$-type-definable for every $k$ and for every closed $C \subseteq S$. By the definition of locally compact models, there is an open  neighborhood $U$ of $0$ in $S$ with $f^{-1}[U] \subseteq X^n$ [resp.  $f^{-1}[U] \subseteq X_n$] for some $n$. Since $\bar f^{-1}[U]$ is a union of $M$-definable sets, we conclude that $\bar f^{-1}[U] \subseteq \bar X^n$ [resp. $\bar f^{-1}[U] \subseteq \bar X_n$]. Hence, 
$\ker (\bar f) \subseteq \bar X^n$ [resp. $\bar X_n$], and so we get that $\ker(\bar f)$ is an $M$-type-definable subgroup [two-sided ideal] of $\RR$ of bounded index (it is bounded by the cardinality of the closure of $f[R]$ in $S$, so by $2^{2^{|R|}}$). Therefore, $\RR^{00}_M \subseteq \ker(\bar f)$, and so $\bar f$ factors through the quotient map $\bar h \colon \RR \to \RR/\RR^{00}_M$, i.e. there is a homomorphism $g \colon  \RR/\RR^{00}_M \to S$ such that $\bar f  = g \circ \bar h$. Since $h \subseteq \bar h$ and $f \subseteq \bar f$, we get $f=g \circ h$. It remains to check that $g$ is continuous. Take any closed $C \subseteq S$. Then $\bar h^{-1}[g^{-1}[C]] = \bar f^{-1}[C]$ has type-definable intersections with all the $\bar X^m$'s [resp. $\bar X_m$'s]. Therefore, $g^{-1}[C]$ is closed by the definition of the logic topology.
\end{proof}

\begin{corollary}\label{corollary: existence of locally compact models}
The existence of a definable locally compact model of $X$ is equivalent to the existence of $\RR^{00}_M$.
\end{corollary}

\begin{proof}
If $\RR^{00}_M$ exists, then the quotient map  $R \to \RR/\RR^{00}_M$ is a definable locally compact model. Conversely, if  $f \colon R \to S$ is a definable locally compact model, then, by the last paragraph of the proof of Proposition \ref{proposition: locally compact model}, $\ker(\bar f)$ is an $M$-type-definable subgroup [two-sided ideal] of $\RR$ of bounded index. Hence, $\RR^{00}_M$ exists.
\end{proof}

The next fact is folklore, but it is important, as it shows that the existence of a locally compact model can be seen as a structural result about the approximate subgroup [subring] in question. It will be also used in the proof of Lemma \ref{lemma: criterion for definable subring}.

By a {\em generic subset} of the subgroup [subring] $\langle X \rangle$ we mean a subset whose finitely many left translates [additive translates] cover $X$.


\begin{fact}\label{fact: preimage of open is generic}
Let $f \colon \langle X \rangle \to S$ be locally compact model of $X$. Then:
\begin{enumerate}
\item $f^{-1}[U]$ is generic for every neighborhood $U$ of $e$ [resp. $0$] in $S$;
\item $f^{-1}[U]$ is [additively] commensurable with $X$ for every relatively compact neighborhood $U$ of $e$ [resp. $0$] in $S$.
\end{enumerate}
\end{fact}

\begin{proof}
Let us focus on the group case. The ring case is almost the same (in (2) one should additionally use Fact \ref{fact: by Lemma 5.5 of Bru}).

(1)  Choose a neighborhood $W$ of $e$ such that $W^{-1}W \subseteq U$.  Since the closure $\cl(f[X])$ is compact, $\cl(f[X]) \subseteq \bigcup_{i<n}g_iW$ for some $n \in \mathbb{N}$ and $g_0,\dots,g_{n-1} \in S$. For every $i<n$ pick $a_i \in f^{-1}[g_iW]$ (if there is any). The computation from the proof of Lemma \ref{lemma: extension of definable map}(2)  shows that $X \subseteq f^{-1}[\cl(f[X])] \subseteq  \bigcup_{i<n}a_i f^{-1}[U]$, so $f^{-1}[U]$ is generic.

(2) Since $\cl(U)$ is compact, the computation  from the proof of Lemma \ref{lemma: extension of definable map}(2)  shows that $f^{-1}[U] \subseteq X^k$ for some $k$. Therefore, $f^{-1}[U]$ is covered by finitely many left translates of $X$. The fact that $X$ is covered by finitely many left translates of $f^{-1}[U]$ follows from (1).
\end{proof}

We finish this section with an observation that if we want to have the existence of a locally compact model, we should work with our definition of an approximate subring.

\begin{remark}
Let $X$ be a symmetric subset of a group [ring], which is not assumed to be an approximate subgroup [approximate subring]. Suppose that there is $f \colon \langle X \rangle \to S$ which satisfies the definition of locally compact model of $X$. Then $X^m$ [resp. $X_m$] is an approximate subgroup [subring] for some $m$. 
\end{remark}

\begin{proof}
Let us do the group case. By definition, there is a neighborhood $U$ of $e$ in $S$ such that $f^{-1}[U] \subseteq X^m$ for some $m$. Since $\cl (f[X])$ is compact, we get that $\cl (f[X^{2m}]) = \cl(f[X])^{2m}$ is compact. Thus, the argument from the proof of Fact \ref{fact: preimage of open is generic}(1) shows that $X^{2m}$ is covered by finitely many left translates of $f^{-1}[U]$. Therefore, $X^{2m}$ is covered by finitely many left translates of $X^m$.
\end{proof}

\section{Generating in $1\frac{1}{2}$ steps and a locally compact model}\label{section: main results}

Here, we prove the main results of this paper, answering the main question from \cite{KrRz} and providing locally compact models for arbitrary definable approximate subrings (so, in particular, abstract approximate subrings by taking the full structure). The goal is to prove:

\begin{theorem}\label{theorem: main theorem}
Let $X$ be a 0-definable (in $M$) approximate subring, $R:=\langle X \rangle$, $\bar R = \langle \bar X \rangle$, and let  $A \subseteq \C$ be a small set of parameters. Then $(\bar R,+)^{00}_A + \bar R  \cdot (\bar R,+)^{00}_A=\RR^{000}_A$. 
Moreover, if $R \subseteq \dcl(A)$, then $\RR^{00}_A$ exists and equals $\RR^{000}_A = (\RR,+)^{00}_A +\bar X (\RR,+)^{00}_A$.
\end{theorem}


First of all, we have

\begin{fact}\label{fact: main consequence of MaWa and Mas}
If $Z$ is a definably amenable 0-definable (in $M$) approximate subgroup, then $\langle \bar Z \rangle^{00}_A$ exists (where $\langle \bar Z \rangle$ is the group generated by $\bar Z$). Moreover,  $\langle \bar Z \rangle^{00}_A \subseteq \bar Z^8$, and if $A=M$, then $\langle \bar Z \rangle^{00}_A \subseteq \bar Z^4$. In particular, $(\RR,+)^{00}_A$ exists and is contained in $8\bar X$, and if $A=M$, then it is contained in $4\bar X$.
\end{fact}

\begin{proof}
The part concerning $Z$ follows from \cite[Theorem 12 or Corollary 13]{MaWa} and \cite[Theorem 5.2]{Mas}. If we work with $A=M$, then instead of \cite[Theorem 5.2]{Mas}, an easy compactness argument from the proof of Claim 1 of \cite{KrPi} in Case 2 (on page 1282) can be used to make sure that the parameters are taken from $M$, and it gives us $\langle \bar Z \rangle^{00}_M \subseteq \bar Z^4$.  For the second part (concerning $R$), note that since the additive group generated by $\bar X$, say $G$, is abelian and so amenable, by Lemma 6.1  of the first arXiv version of \cite{Hru2} or \cite[Proposition 5.8]{Ma}, we get that $\bar X$ is an amenable (and so definably amenable) approximate subgroup, hence $G^{00}_A$ exists and satisfies the desired inclusions by the first part. Then $G^{00}_A = (\RR,+)^{00}_A$, because $G$ is of bounded (even countable) index in $(\RR,+)$ by Fact \ref{fact: by Lemma 5.5 of Bru}.
\end{proof}

We will need the notion of thick subset of $\bar R$, as given in \cite[Definition 4.1]{HKP}. 

\begin{definition}\label{definition: thick}
A definable, additively symmetric subset $D$ of $\bar R$ is {\em thick} if for every sequence $(r_i)_{i<\lambda}$ of unbounded length which consists of elements of $\bar R$ there are $i<j<\lambda$ with $r_j -r_i \in D$.
\end{definition}

Using compactness, one gets 

\begin{remark}\label{remark: thick} A definable, additively symmetric subset $D$ of $\RR$ is thick if and only if  for every $m \in \omega$ there exists a positive integer $M$ such that for every $r_0,\dots,r_{M-1} \in \bar X_m$ there are $i<j<M$ with $r_j -r_i\in D$. For any $M$ with this property, 
we will say that $D$ is {\em $M$-thick} in $\bar X_m$.
\end{remark}

Using this remark together with finite Ramsey theorem (exactly as in the proof of \cite[Lemma 1.2]{Gis2}), we get that the class of thick subsets of $\bar R$ is closed under finite intersections. Remark \ref{remark: thick} also implies that in Definition \ref{definition: thick} the adjective ``unbounded'' can be replaced by ``uncountable''. 

The following basic observation will be crucial in the proof of the main lemma below. 
From now on, in this section, $H:=(\bar R,+)^{00}_A$.

\begin{remark}\label{remark: intersection of thick sets}
Every definable, additively symmetric subset of $\bar R$ which contains $H$ is thick. Thus, $H$ is the intersection of a downward directed family of $A$-definable thick subsets of $\RR$.
\end{remark}

\begin{proof}
Since $[\RR: H] \leq 2^{|\mathcal{L}|+|A|}$, we have that for any $\lambda >2^{|\mathcal{L}|+|A|}$, for every sequence $(r_i)_{i<\lambda}$ of elements of $\bar R$ there are $i<j<\lambda$ with $r_j -r_i \in H$. Hence, the same is true for any superset of $H$, and so all definable, additively symmetric supersets of $H$ are thick. The second part follows from that, since $H$ is clearly the intersection of the family of all  $A$-definable, additively symmetric subsets of $\RR$ containing $H$ and this family is downward directed.
\end{proof}

We will also  need the following definition and remark from \cite{KrRz}.

	\begin{definition}
		\label{def:coset_indep}
		We will say that two subgroups $H_1$ and $H_2$ of an abelian group $G$ are \emph{coset-independent} if any coset of $H_1$ intersects any coset of $H_2$. They are \emph{coset-dependent} if they are not coset-independent.
	\end{definition}
	
	\begin{remark}\label{remark: coset independent}
		Let $G$ be an abelian group and $H_1,H_2 \leq G$. The following conditions are equivalent.
		\begin{enumerate}[label=(\roman*)]
			\item $H_1$ and $H_2$ are coset-independent.
			\item $H_1$ intersects any coset of $H_2$.
			\item $H_1+H_2 =G$.
		\end{enumerate}
		Thus, $H_1$ and $H_2$ are coset-dependent if and only if $H_1+H_2$ is a proper subgroup of $G$.
	\end{remark} 

The next lemma is the technical core of the proof of Theorem \ref{theorem: main theorem}. This lemma is a variant of Lemma 4.4 from \cite{KrRz}, and its proof is a non-trivial elaboration on the proof of that lemma.

\begin{lemma}\label{lemma: main technical lemma}
Let $G$ be the intersection of all sets of the form $\RR K/H$, where $K$ ranges over all bounded index subgroups of $(\RR,+)$ which are type-definable over some 
sets of parameters of cardinality at most $2^{2^{|\mathcal{L}| +|A|}}$. Then $G$ is a subgroup of $(\RR/H, +)$.
\end{lemma}

\begin{proof}
It is clear that $0/H \in G$ and $G$ is closed under additive inverses. Thus, we need to show that it is closed under $+$. So consider any $a,b \in G$, and we will show that $a+b \in G$.

The family of subgroups of $(\bar R, +)$ over which $K$ ranges in the statement of the lemma will be denoted by $\mathcal{K}$.

By Remark \ref{remark: intersection of thick sets}, $H = \bigcap_{i \in I} D_i$ for some downward directed family $\{D_i\}_{i \in I}$ of $A$-definable thick subsets of $\bar R$. We can assume that $|I| \leq |\mathcal{L}| + |A|$. Using Remark \ref{remark: thick} and the terminology introduced there, for every $i \in I$ and $m \in \omega$ we can choose a positive integer $M_{i,m}$ such that  $D_i$ is $M_{i,m}$-thick in $\bar X_m$.

For every $i \in I$, $s \in \RR$, and $K \in \mathcal{K}$, define:

$$ n_{i,s,K} := \max \{ |Y| : Y \subseteq sK \textrm{ and for every distinct } x,y \in Y \textrm{ we have } x-y \notin D_i\}.$$

Note that since $sK \subseteq \bar X_k$ for some $k \in \omega$, we have $n_{i,s,K} < M_{i,k}$. Since for any fixed $m \in \omega$ and $K \in \mathcal{K}$ there is $k$ such that $\bar X_m K \subseteq \bar X_k$, for every fixed $i \in I$ and $m \in \omega$ there exists a smallest $n_{i,m} \in \omega$ for which there exists $K_{i,m} \in \mathcal{K}$ such that

$$(\forall s \in \bar X_m)(b \in sK_{i,m}/H \Rightarrow n_{i,s,K_{i,m}} \leq n_{i,m}).$$
(In particular, if there is no $s \in \bar X_m$ for which $b \in sK_{i,m}/H$, then $n_{i,m}=0$.)
Put 

$$K_{I,\omega}:= \bigcap_{i \in I} \bigcap_{m \in \omega} K_{i,m}.$$

Since $|I| \leq |\mathcal{L}| + |A|$, we get $K_{I,\omega} \in \mathcal{K}$. By the above choices, we also have

$$(*) \;\;\;\;\;\; (\forall i \in I)(\forall m \in \omega)(\forall s \in \bar X_m) (b \in sK_{I, \omega}/H \Rightarrow n_{i,s,K_{I, \omega}} \leq n_{i,m}).$$

For $r \in \RR$ let $g_r \colon \RR \to \RR/H$ be given by $g_r(x) := rx/H$. It is a group homomorphism. Note that $[\RR: \ker(g_r)] \leq |\RR/H| \leq 2^{|\mathcal{L}| + |A|}$.

		\paragraph{\textbf{Case 1.}} For every $K \in \mathcal{K}$ with $K \leq K_{I,\omega}$, there are $r,s \in \RR$ with $a \in rK/H$ and $b \in sK/H$ such that $\ker(g_r) \cap K$ and $\ker(g_s) \cap K$ are coset-independent subgroups of $K$.

Then, since $g_r^{-1}(a) \cap K$ and $g_s^{-1}(b) \cap K$ are cosets of  $\ker(g_r) \cap K$ and $\ker(g_s) \cap K$, respectively, they have a non-empty intersection, i.e. there is $k \in K$ with $rk/H=a$ and $sk/H=b$.
		Hence, $a+b = (r+s)k/H \in \RR K/H$. Since this holds for every $K \in \mathcal{K}$ with $K \leq K_{I,\omega}$ (so also for every  $K \in \mathcal{K}$), we conclude that $a+b \in G$.

		\paragraph{\textbf{Case 2.}} There exists $K \in \mathcal{K}$ with $K \leq K_{I,\omega}$ such that for all $r,s \in R$ with $a \in rK/H$ and $b \in sK/H$, $\ker(g_r) \cap K$ and $\ker(g_s) \cap K$ are coset-dependent subgroups of $K$.

By the definition of $G$, pick $r_0 \in \bar R$ with $a \in r_0K/H$. By Remark \ref{remark: coset independent}, for any $s \in \RR$ with $b\in sK/H$ (by the definition of $G$, at least one such $s$ exists),
		
$$(**)\;\;\;\;\;\; \ker(g_{r_0}) \cap K \leq (\ker(g_{r_0}) \cap K) + (\ker(g_s) \cap K) \lneq K.$$

Put $L_s:= (\ker(g_{r_0}) \cap K) + (\ker(g_s) \cap K)$. Since $[K : \ker(g_{r_0}) \cap K] \leq |\RR/H| \leq 2^{|\mathcal{L}|+|A|}$, there are at most $2^{2^{|\mathcal{L}|+|A|}}$ possibilities for $L_s$ when $s$ varies as above. Let $K_{I,\omega}'$ be the intersection of all these $L_s$'s. Since each $L_s$ is type-definable over the parameters over which $K$ is defined together with 
$r_0,s$, we see that $L_s \in \mathcal{K}$. Hence, $K_{I,\omega}' \in \mathcal{K}$.

Since there are at most $2^{2^{|\mathcal{L}|+|A|}}$ possibilities for $L_s$, by (**), there exists a set $E \subseteq K$ with $|E| \leq 2^{2^{|\mathcal{L}|+|A|}}$ such that 

$$(***)\;\;\;\;\;\; (\forall s \in \RR)(b \in sK/H \Rightarrow (\exists k \in E)(k \in K \setminus L_s)).$$
Note that the condition $k \in K \setminus L_s$ implies that $sk \notin sL_s +H$.

\begin{clm*}
For every $m \in \omega$ there exists $i_m \in I$ such that for every $s \in \bar X_m$ with $b \in sK/H$ there is $k \in E$ such that $sk \notin sL_s +D_{i_m}$.
\end{clm*}

\begin{clmproof}
Suppose this fails, which is witnessed by some $m \in \omega$. Note that the sets $L_s$ are type-definable uniformly in $s$, that is there is a type $\pi(x,y)$ (with some fixed parameters) such that $L_s = \pi(\C, s)$ for every $s$ in question. Thus, since $E$ is small and $\{D_i\}_{i \in I}$ is downward directed, by compactness (or rather saturation of $\C$), there exists $s \in \bar X_m$ with $b \in sK/H$ such that $(\forall  k \in E)(sk \in sL_s +H)$, a contradiction with $(***)$.
\end{clmproof}



Now, by the definition of $G$, we can find $s_0 \in \RR$ for which $b \in s_0K_{I,\omega}'/H$. Choose $m<\omega$ such that $s_0 \in \bar X_m$. By the claim, we get 

$$(\forall s \in \bar X_m) ( b \in sK/H \Rightarrow n_{i_m,s, L_s} < n_{i_m, s, K}).$$
Since $K_{I,\omega}' \leq L_s \leq K \leq K_{I,\omega}$, we conclude that

$$(\forall s \in \bar X_m) (b \in sK_{I,\omega}'/H \Rightarrow n_{i_m,s, K_{I,\omega}'} < n_{i_m, s, K_{I,\omega}}).$$
By $(*)$, this implies that

$$(\forall s \in \bar X_m) (b \in sK_{I,\omega}'/H \Rightarrow n_{i_m,s, K_{I,\omega}'} < n_{i_m, m}),$$
which contradicts the minimality of $n_{i_m,m}$ (as there is at least one $s \in \bar X_m$ with $b \in sK_{I,\omega}'/H$, namely $s_0$).
\end{proof}

In order to prove Theorem \ref{theorem: main theorem}, we need one more non-trivial ingredient stated below. For a proof in the context of definable groups see the proof of Fact 2.2 of \cite{KrRz}. It works the same for definable approximate subgroups, as the facts on which it relies (i.e. \cite[Theorem 12]{MaWa} and \cite[Theorem 5.2]{Mas}) are stated for definable approximate subgroups. Also, instead of Lemmas 2.2(2) and 3.3 of \cite{Gis1}, 
one should use their versions for approximate subgroups  stated in Propositions 4.3 and 4.5 of \cite{HKP}.

\begin{fact}\label{fact: 00=000 for def. amen.}
If $Z$ is a definably amenable 0-definable approximate subgroup, then $\langle \bar Z \rangle^{00}_A = \langle \bar Z \rangle^{000}_A$. In particular, $(\RR,+)^{00}_A = (\RR,+)^{000}_A$.
\end{fact}

Now, we are ready to prove Theorem \ref{theorem: main theorem}.

\begin{proof}[Proof of Theorem \ref{theorem: main theorem}]
Let $\mathcal{K}$ be the family of all bounded index subgroups of $(\RR,+)$ which are type-definable over some sets of parameters of cardinality at most $2^{2^{|\mathcal{L}| +|A|}}$. Let $G := \bigcap_{K \in \mathcal{K}} \RR K/H$, as in Lemma \ref{lemma: main technical lemma}. By Lemma \ref{lemma: main technical lemma}, $G$ is a subgroup of $\RR/H$ which is clearly $A$-invariant. 

Since $|\RR/H| \leq 2^{|\mathcal{L}| +|A|}$, $G$ is an intersection of at most $2^{2^{|\mathcal{L}| +|A|}}$ sets $\RR K/H$, i.e. $G=\bigcap_{K \in \mathcal{K}_0} \RR K/H$ for some $\mathcal{K}_0 \subseteq \mathcal{K}$ of cardinality bounded by  $2^{2^{|\mathcal{L}| +|A|}}$. Let $K_0: = H \cap  \bigcap \mathcal{K}_0$. Then $K_0 \in \mathcal{K}$ and $G=\RR K_0/H$. 

Let $H_0:=\bigcap_{r \in \RR} g_r^{-1}[G]$, where $g_r \colon H \to \RR/H$ is given by $g_r(x) := rx/H$. Since $G$ is a subgroup of $\RR/H$ and $G$ and $H$ are both  $A$-invariant, we get that $H_0$ is an  $A$-invariant subgroup of $H$. It is clear that $K_0 \leq H_0$, so $H_0$ is of  bounded index in $(\RR,+)$. Therefore, $(\RR,+)^{000}_A \leq H_0 \leq H = (\RR,+)^{00}_A$. Since by Fact \ref{fact: 00=000 for def. amen.} $(\RR,+)^{00}_A =(\RR,+)^{000}_A$, we conclude that $H_0=H$. 

On the other hand, $\RR H_0/H =G$ (($\subseteq$) follows by the definition of $H_0$, while ($\supseteq$) follows from the fact that $K_0 \leq H_0$ and $\RR K_0/H =G$).

Putting the last two paragraphs together, we get $G=\RR H/H$. Hence, $\pi^{-1}[G] = H + \RR H$ is an $A$-invariant, bounded index subgroup of $\RR$, where $\pi \colon \RR \to \RR/H$ is the quotient map. It follows that $H + \RR H$ is closed under left multiplication by the elements of $\RR$, and so it is a left ideal. Therefore, by Proposition \ref{proposition: side is irrelevant} and Fact \ref{fact: 00=000 for def. amen.}, we get 
$$\RR^{000}_A \subseteq H + \RR H = (\RR,+)^{000}_A + \RR (\RR,+)^{000}_A \subseteq \RR^{000}_A + \RR \RR^{000}_A = \RR^{000}_A,$$
and hence  $H + \RR H = \RR^{000}_A$ as required.

For the ``moreover'' part, it is enough to prove that

$$(\RR,+)^{00}_A + \RR(\RR,+)^{00}_A = (\RR,+)^{00}_A +\bar X (\RR,+)^{00}_A.$$

Indeed, since the right hand side is $A$-type-definable and the left hand side equals  $\RR^{000}_A$ by the first part of the theorem, we conclude that 
both sides are equal to $\RR^{00}_A$, which will complete the proof.

The inclusion ($\supseteq$) is obvious. So we prove ($\subseteq$). By Fact \ref{fact: by Lemma 5.5 of Bru}, we can choose a countable subset $Y$ of $R$ so that $Y +\bar X =\RR$. It is enough to show that

$$(\forall y \in Y)( y(\RR,+)^{00}_A \subseteq (\RR,+)^{00}_A).$$

So pick $y \in Y$. Let $l_y \colon  \RR \to \RR$ be given by $l_y(t) := yt$. This is a $\{y\}$-invariant group homomorphism. Choose $m \in \omega$ so that $(\RR,+)^{00}_A \subseteq \bar X_m$. Since $l_y|_{\bar X_m}$ is $\{y\}$-definable, and, by the assumption that $R \subseteq \dcl(A)$ we have $y \in \dcl(A)$, we get $l_y[(\RR,+)^{00}_A] = (l_y[\RR],+)^{00}_A$. 
On the other hand, $(l_y[\RR],+)^{00}_A \leq  (\RR,+)^{00}_A$, for if not, then $(l_y[\RR],+)^{00}_A \cap (\RR,+)^{00}_A$ would be a proper $A$-type-definable subgroup of $(l_y[\RR],+)^{00}_A$ of bounded index.
Therefore, $y(\RR,+)^{00}_A =l_y[(\RR,+)^{00}_A]= (l_y[\RR],+)^{00}_A \subseteq  (\RR,+)^{00}_A$.
\end{proof}


The next corollary answers positively Question 1.3 of \cite{KrRz}.

\begin{corollary}
If $\bar R$ is definable, then $(\bar R,+)^{00}_A + \bar R  \cdot (\bar R,+)^{00}_A=\RR^{000}_A = \RR^{00}_A$ for an arbitrary small $A \subseteq \C$.
\end{corollary}

\begin{proof}
It follows from Theorem  \ref{theorem: main theorem}, because  $(\bar R,+)^{00}_A + \bar R  \cdot (\bar R,+)^{00}_A$ is $A$-type-definable.
\end{proof}

The next corollary yields the existence and a description of the universal definable locally compact model for an arbitrary definable approximate subring.

\begin{corollary}\label{corollary: locally compact model exists}
Let $X$ be a 0-definable (in $M$) approximate subring, $R:=\langle X \rangle$, and $\bar R = \langle \bar X \rangle$. Then $X$ has a definable locally compact model. More precisely, the quotient map  $h \colon R \to \RR/\RR^{00}_M$ is the universal definable locally compact model of $X$, and $U:= \{ a/\RR^{00}_M : a +\RR^{00}_M \subseteq 4\bar X + \bar X \cdot 4\bar X\}$ is an open neighborhood of $0/\RR^{00}_M$ such that $h^{-1}[U] \subseteq 4X + X\cdot 4X$.
\end{corollary}

\begin{proof}
By the ``moreover'' part of Theorem  \ref{theorem: main theorem} and Proposition \ref{proposition: locally compact model}, we get that $\RR^{00}_M$ exists and the quotient map  $h \colon R \to \RR/\RR^{00}_M$ is the universal definable locally compact model of $X$. 

By Fact \ref{fact: main consequence of MaWa and Mas}, we know that $(\RR,+)^{00}_M \subseteq 4\bar X$. So, by the ``moreover'' part of Theorem  \ref{theorem: main theorem}, we get $\RR^{00}_M =  (\RR,+)^{00}_M +\bar X (\RR,+)^{00}_M \subseteq 4\bar X + \bar X \cdot 4\bar X$. Thus, $0/\RR^{00}_M \in U$. The fact that $U$ is open follows easily from the definition of the logic topology on $\RR/\RR^{00}_M$. The fact that $h^{-1}[U] \subseteq 4X + X\cdot 4X$ is obvious by the definition of $U$.
\end{proof}

\section{Applications}\label{section: application}

In this section, we give an application of the existence of definable locally compact models to a classification up to additive commensurability of all approximate subrings of rings of positive characteristic as well as all finite approximate subrings of rings without zero divisors.

Let again $X$ be a 0-definable (in $M$) approximate subring, $R:=\langle X \rangle$, $\bar R = \langle \bar X \rangle$. By Corollary \ref{corollary: locally compact model exists}, we can choose  $f \colon R \to S$ a definable locally compact model of $X$ with dense image; let $\bar f \colon \bar R \to S$ be the unique extension of $f$ as in Lemma \ref{lemma: extension of definable map}(1).  (For example, one can take $S:=\RR/\RR^{00}_M$ and $\bar f$ the quotient map $\bar R \to  \RR/\RR^{00}_M$.)  We leave as an exercise to show that $\bar f$ is onto.

Although only the implication $(\leftarrow)$ in the next lemma will be used in the proof of Theorem \ref{theorem: application of the main theorem}, for completeness we state the lemma in the form of an equivalence.

\begin{lemma}\label{lemma: criterion for definable subring}
$X$ is additively commensurable with a definable subring of $R$ if and only if $S$ has a compact open subring. More precisely, if $U$ is a compact open subring of $S$, then $f^{-1}[U]$ is a definable subring of $R$ additively commensurable with $X$.
\end{lemma}

\begin{proof}
$(\rightarrow)$ Suppose $X$ is additively commensurable with a definable subring $P$. Then the same is true about $\bar X$ and $\bar P$. Hence, $\bar P$ is of bounded (even countable) index in $\bar R$, and so, by virtue of Proposition \ref{proposition: side is irrelevant},  $\RR^{00}_M \subseteq \bar P$. Therefore, $\bar h^{-1}[\bar h[\bar P]] = \bar P$, and since $\bar P$ is definable, we get that $\bar h[P]$ is a clopen compact subring of  $\RR/\RR^{00}_M$ (where $\bar h \colon \RR \to \RR/\RR^{00}_M$ is the quotient map), which completes the proof when $S= \RR/\RR^{00}_M$. In the general case, by the last paragraph of the proof of Proposition \ref{proposition: locally compact model}, $\bar f$ factors through $\bar h$ via a continuous epimorphism $g \colon \RR/\RR^{00}_M \to S$, i.e. $\bar f = g \circ \bar h$, 
and the conclusion follows from the open mapping theorem applied to the $\sigma$-compact topological group $\RR/\RR^{00}_M$ and the continuous epimorphism $g$ with the image being the locally compact group $(S,+)$. 


$(\leftarrow)$ Let $U \subseteq  S$ be a compact open subring. 
Since $f$ is definable and $U$ is compact and open, we get that $f^{-1}[U]$ is a definable subring of $R$. On the other hand, by Fact \ref{fact: preimage of open is generic}(2), it is additively commensurable with $X$.
\end{proof}

\begin{theorem}\label{theorem: application of the main theorem}
If $R$ is of positive characteristic, then $X$ is additively commensurable with a definable subring of $R$ contained in $4X +X \cdot 4X$.
\end{theorem}

\begin{proof}
Since $R$ is of positive characteristic and $f[R]$ is dense in $S$, the ring $S$ is also of positive characteristic. Thus, $(S,+)$ is a torsion, locally compact abelian group, and as such it has a basis of neighborhoods of $0$ consisting of compact open (so clopen) subgroups  (see \cite[Theorem 3.5]{Arm}). This implies that it has a basis of neighborhoods of $0$ consisting of compact open subrings (but not necessarily ideals). Indeed, for every compact open subgroup $V$ of $(S,+)$ there exists an open subset $U \ni 0$ of $V$ such that $UV \subseteq V$. Then the subring $\langle U \rangle$ generated by $U$ is open and contained in $V$; hence,  $\langle U \rangle$ is clopen and so compact.

We have shown that $S$ has a compact open subring, so the existence of a definable subring of $R$ additively commensurable with $X$ follows from Lemma \ref{lemma: criterion for definable subring}$(\leftarrow)$.

To get such a subring which is additionally contained in $4X +X \cdot 4X$, let us work with $S:=  \RR/\RR^{00}_M$. By Corollary \ref{corollary: locally compact model exists}, there is an open neighborhood of $0$ in $S$ with the preimage under the quotient map $h\colon R \to \RR/\RR^{00}_M$ contained in  $4X +X \cdot 4X$. By the first paragraph of this proof, we can find a smaller neighborhood $U$ of $0$ which is a compact open subring. By the more precise information in Lemma \ref{lemma: criterion for definable subring}, $h^{-1}[U]$ is a definable subring of $R$ additively commensurable with $X$, and it is clearly contained in $4X +X \cdot 4X$.
\end{proof}

The following application of the above theorem uses a standard ultraproduct argument.

\begin{corollary}\label{corollary: application to finite approximate subrings}
For every $K,L \in \mathbb{N}$ there exists a constant $C(K,L) \in \mathbb{N}$ such that every $K$-approximate subring $X$ of a ring of positive characteristic $\leq L$ is additively $C(K,L)$-commensurable with a subring contained in $4X +X \cdot 4X$.
\end{corollary}

\begin{proof}
Suppose for a contradiction that for every $n \in \omega$ there is a $K$-approximate subring $X_n$ of a ring of positive characteristic $\leq L$ which is not additively $n$-commensurable with a subring contained in $4X_n +X_n \cdot 4X_n$. Let $R_n:= \langle X_n \rangle$, all considered in the language of rings expanded by an additional predicate symbol $P$ interpreted in $R_n$ as $P(R_n):=X_n$ . Take a non-principal ultrafilter $\mathcal{U}$ on $\omega$, and let $M := \prod R_n/\mathcal{U}$. Then $M$ is a ring of positive characteristic $\leq L!$. Let $X:= P(M)=\prod X_n/\mathcal{U}$. It is a definable $K$-approximate subring of $M$. By Theorem \ref{theorem: application of the main theorem}, there exists a definable subring $F$ of $M$ which is additively $m$-commensurable with $X$ for some $m \in \omega$ and contained in $4X + X \cdot 4X$. Then $F=\varphi(M,\bar a)$ for some formula $\varphi(x,\bar y)$ over $\emptyset$ and tuple $\bar a$ of parameters from $M$; so $\bar a = ((a_{1n})_{n<\omega}/\mathcal{U},\dots, (a_{kn})_{n<\omega}/\mathcal{U})$. We conclude that for all $n$ from some set $U \in \mathcal{U}$ we have that $\varphi(R_n, (a_{1n},\dots,a_{kn}))$ is a subring of $R_n$ which is additively $m$-commensurable with $X_n$ and contained in $4X_n +X_n \cdot 4X_n$. This yields a contradiction by taking any $n \geq m$ from $U$.
\end{proof}

Note that both in Theorem \ref{theorem: application of the main theorem} and Corollary \ref{corollary: application to finite approximate subrings} one cannot require that there is a ring commensurable with $X$ which contains $X$. As an example, take the finite approximate subring consisting of all linear polynomials inside the ring of polynomials $\mathbb{F}_p[t]$ over the finite field $\mathbb{F}_p$. It is also clear that Theorem \ref{theorem: application of the main theorem} would fail if we dropped the positive characteristic assumption. As an example, take the approximate subring $[-1,1]$ in the field of reals. The problem of classifying approximate subrings in the zero characteristic case remains open. But even in positive characteristic, although Theorem \ref{theorem: application of the main theorem} classifies up to additive commensurability all approximate subrings as subrings, having in mind that an additively  symmetric subset (of a given ring) additively commensurable with a subring need not be an approximate subring (see the last paragraph of Subsection \ref{subsection: approximate rings}), it remains open to fully classify all approximate subrings of rings of positive characteristic.

Now, we classify finite approximate subrings of rings without zero divisors.

\begin{theorem}\label{theorem: classification without zero divisors}
For every $K \in \mathbb{N}$ there exists $N(K) \in \mathbb{N}$ such that for every  finite $K$-approximate subring $X$ of a ring without zero divisors either $|X| <N(K)$ or $4X + X \cdot 4X$ is a subring which is additively $K^{11}$-commensurable with $X$.
\end{theorem}

\begin{proof}


Suppose for a contradiction that for every $n \in \omega$ there is a finite $K$-approximate subring $X_n$ of size $\geq n$ of a ring without zero divisors such that $4X_n + X_n \cdot 4X_n$ is not a subring. Let $R_n:= \langle X_n \rangle$, all considered in the language of rings expanded by an additional predicate symbol $P$ interpreted in $R_n$ as $P(R_n):=X_n$. Take a non-principal ultrafilter $\mathcal{U}$ on $\omega$, and let $M := \prod R_n/\mathcal{U}$. Then $M$ is a ring without zero divisors.   Let $X:= P(M)=\prod X_n/\mathcal{U}$. It is an infinite, definable $K$-approximate subring of $M$. Let $R :=\langle X \rangle$. Pass to a monster model $\C \succ M$. By Corollary \ref{corollary: locally compact model exists}, $\bar R^{00}_M$ is an ideal of $\bar R$ contained in $\bar Y:=4\bar X + \bar X \cdot 4\bar X$. Since $X$ is infinite, $\bar X$ is of unbounded cardinality, and so $\bar R^{00}_M \ne \{0\}$. Pick any non-zero $a \in \bar R^{00}_M$. Then $a(\bar Y \cdot \bar Y + (\bar Y + \bar Y)) \subseteq \bar R^{00}_M \subseteq \bar Y$. So there is a non-zero $b \in R$ such that $b(Y\cdot Y +(Y+Y)) \subseteq Y$, where $Y:=4X + X \cdot 4X$. Then there is $n \in \mathbb{N}$ for which there exists a non-zero $c \in R_n$ such that for $Y_n:= 4X_n + X_n \cdot 4X_n$ we have $c(Y_n \cdot Y_n +(Y_n+Y_n)) \subseteq Y_n$. However, since $Y_n$ is additively symmetric and not a subring, we have that $Y_n \cdot Y_n +(Y_n+Y_n)$ is a proper superset of $Y_n$. Using finiteness of $Y_n$, we conclude that $c$ is a zero divisor, a contradiction. 

The fact that $K^{11}$ additive translates of $X$ cover $4X + X \cdot 4X$ is an easy computation.
\end{proof}

In the proof, we have that $a \in \bar Y$, so we can choose $b \in Y$ and $c \in Y_n$. Therefore, the assumption of the last theorem can be weakened to requiring that there are no zero divisors in $Y:=4X + X \cdot 4X$ witnessed by an element from $2(Y \cdot Y + (Y +Y))$.

\section{The 0-components}\label{section: application}

In the case of a 0-definable (in the monster model $\C$) ring $\bar R$, we have Fact \ref{fact: ideal component} for $\RR^0_A$, and we know by Corollary 2.10 of \cite{KrRz} that $\RR^{00}_A =\RR^{0}_A$ whenever $\RR$ is unital or of positive characteristic. In particular, in those two cases, $\RR/\RR^{00}_A=\RR/\RR^0_A$ is a profinite ring. In this subsection, we explain that all of this drastically fails for approximate subrings.

A natural counterpart of $\RR^0_A$ for approximate subrings is as follows. From now on, let $X$ be a 0-definable (in $M$) approximate subring, $R:=\langle X \rangle$, $\bar R = \langle \bar X \rangle$, and let  $A \subseteq \C$ be a small set of parameters.

\begin{definition}
${\bar R}^0_{A,ideal}$ is the intersection of all $A$-definable two-sided ideals of $\bar R$ of countable (equivalently, bounded) index.
${\bar R}^0_{A,ring}$ is the intersection of all $A$-definable subrings of $\bar R$ of countable index.
\end{definition}

The existence of ${\bar R}^0_{A,ideal}$ [resp. ${\bar R}^0_{A,ring}$] is clearly equivalent to the existence of some $A$-definable two-sided ideal [resp. subring] of countable index (equivalently, commensurable with $\bar X$).
We will see in the examples below that it may happen that ${\bar R}^0_{A,ring}$ does not exist as well as that it exists but ${\bar R}^0_{A,ideal}$ does not.  Even $(\RR,+)^0_A$ need not exist. 
Moreover, even for unital $\RR$, $\RR/\RR^{00}_A$ need not be totally disconnected. If $\RR$ is of positive characteristic, then $\RR/\RR^{00}_A$ is totally disconnected  
but need not have a basis of neighborhoods of $0$ consisting of open ideals.
Let us go to some details.

Note that ``$\RR^0_{A,ideal}$ exists'' implies ``$\RR^0_{A,ring}$ exists'' implies ``$(\RR,+)^0_A$ exists''.

\begin{example}
Let $M := (\mathbb R,+,\cdot,0,1)$ and $X:= [-1,1]$ which is clearly a 0-definable approximate subring (and here $R=\R$). Then $(\RR,+)^0_M$ does not exist. Also, $\RR^{00}_M = (\RR,+)^{00}_M = \bigcap_{n \in \omega} \bar I_n =:\mu$, where  $I_n:=[-\frac{1}{n},\frac{1}{n}]$ and $\bar I_n$ is the interpretation of $I_n$ in $\C$  (i.e. $\mu$ is the subgroup of the infinitesimals of $\RR$), and $\RR/\RR^{00}_M$ is isomorphic to $\R$ as a topological ring, so it is not totally disconnected.
\end{example}

\begin{proof}
By compactness, the fact that $(\RR,+)^0_M$ does not exist is equivalent to the fact that there is no definable subgroup of $(R,+)$ contained in some $nX =[-n,n]$ and whose finitely many additive translates cover $nX$. And the right hand side clearly holds, as the only subgroup of $(R,+)$ contained in some $[-n,n]$ is $\{0\}$. 

For the second part, the analysis of Example 3.2 of \cite{GJK} applies with minor adjustments (note that still we have a well-defined standard part map $\st \colon \RR \to \R$ and we show that $\ker(\st) = \RR^{00}_M$).
\end{proof}

If one prefers to work in the abstract context, one can equip the reals with the full structure (where all subsets of all finite Cartesian powers are added as predicates on $M$). Then, by the same reason as above, $(\RR,+)^0_M$ does not exist. Regarding the second part, we get a continuous epimorphism from $\RR/\RR^{00}_M$ to $\R$ which implies that $\RR/\RR^{00}_M$ is not totally disconnected.

\begin{example}\label{example: Laurent series}
Let $M:=\mathbb{F}_p((t))$ be the field of formal Laurent series (over the finite field $\mathbb{F}_p$) equipped with the full structure. 
Let $X$ be the subset (in fact, additive subgroup) consisting of the series of the form $\sum_{i=-1}^\infty a_i t^i$. This is clearly a 0-definable approximate subring, and $R:=\langle X \rangle =\mathbb{F}_p((t))$. Then $\RR^0_{M,ideal}$ does not exist, while $\RR^{0}_{M, ring}$ does exist. The ring $\RR/\RR^{00}_M$ is totally disconnected but does not have a basis of neighborhoods of  0 consisting of open ideals.
\end{example}

\begin{proof}
The existence of $\RR^0_{M,ideal}$ is equivalent to the existence of a definable ideal of $R$ contained in some $X_m$ and whose finitely many additive translates cover $X_m$. But $R$ is a field, so it does not have such an ideal. Thus, $\RR^0_{M,ideal}$ does not exist. On the other hand, since the set $\mathbb{F}_p[[t]]$ of all formal power series is a definable subring of $R$ whose $p$ translates cover $X$, we get that  $\RR^{0}_{M, ring}$ exists. 
Since the additive group of $\RR/\RR^{00}_M$ is a torsion, locally compact abelian group, we get that it is totally disconnected (e.g. see \cite[Theorem 3.5]{Arm}). Let $\st \colon \RR \to R$ be the standard part map (where $R$ is equipped with the usual valuation topology which makes it a locally compact field). Then $\ker(\st) = \bigcap_{n \in \omega} \bar I_n$, where $I_n$ is the set of formal power series of the form $\sum_{i=n}^\infty a_i t^i$. We have that $\RR/\ker(\st)$ is topologically isomorphic to $R$, and the obvious map $\RR/\RR^{00}_M \to \RR/\ker(\st)$ is a continuous epimorphism. Thus, if $\RR/\RR^{00}_M $ had a basis of neighborhoods of  0 consisting of open ideals, then the images of these ideals would form a basis at 0 in $R$ consisting of open ideals, which is a non-sense as $R$ is a non-discrete field.
\end{proof}

Theorem 1.1 of \cite{KrRz} tells us that if $\RR$ is definable and $H$ is an $A$-definable subgroup of $(\RR,+)$ of finite index, then $H + \RR H$ contains an $A$-definable two-sided ideal of finite index. From that it is deduced that $(\RR,+)^0_A + \RR (\RR,+)^0_A = \RR^0_A$ (see Theorem 1.5 and Proposition 3.4(2) of \cite{KrRz}). In our general context of $\bar R$ generated by $\bar X$, both observations fail: Example \ref{example: Laurent series} is a counter-example to both statements (where in the first statement we replace ``finite index'' by ``countable index'', and in the second one we assume that $(\RR,+)^0_A$ exists) . However, the following remains unclear.

\begin{question}
Suppose $(\RR,+)^0_A$ exists. Is it true that $(\RR,+)^0_A + \RR (\RR,+)^0_A$ is a subgroup of $(\RR,+)$?
\end{question}

Recall that Corollary 8.11 of \cite{KrRz} yields an example of a definable, commutative, unital ring $\RR$ and a 0-type-definable subgroup $H$ of $(\RR,+)$ which is an intersection of a countable descending sequence of definable subgroups of finite index (so $(\RR,+)^0_\emptyset \leq H$), but $\RR H$ does not additively generate a subgroup in finitely many steps; in particular, $H + \RR H$ is not a subgroup.

\begin{proposition}
Assume that $\RR$ is of positive characteristic. Then $(\RR,+)^0_A$ exists and coincides with $(\RR,+)^{00}_A$. 
Thus,  $(\RR,+)^0_A + \RR (\RR,+)^0_A = \RR^{000}_A$ (is a subgroup of $(\RR,+)$), and if also $R \subseteq \dcl(A)$, then $(\RR,+)^0_A + \RR (\RR,+)^0_A = \RR^{00}_A$.
\end{proposition}

\begin{proof}
$\RR/(\RR,+)^{00}_A$ is a torsion, locally compact abelian group, and as such it has a basis $\{H_i\}_{i \in I}$ of neighborhoods of $0$ consisting of open (so clopen) subgroups (see \cite[Theorem 3.5]{Arm}). Take $m$ such that $(\RR,+)^{00}_A \subseteq \bar X_m$ (e.g. $m=3$ works, but it does not matter here). Then $U:=\{a/(\RR,+)^{00}_A: a+ (\RR,+)^{00}_A \subseteq \bar X_m\}$ is an open neighborhood of $0$ in $\RR/(\RR,+)^{00}_A$, and without loss of generality we can assume that each $H_i$ is contained in $U$. Let $\pi \colon  \RR \to \RR/(\RR,+)^{00}_A$ be the quotient map. Then both $\pi^{-1}[H_i] \subseteq \bar X_m$ and its complement in $\bar X_m$ are type-definable and so definable sets. Hence, $\pi^{-1}[H_i] \subseteq \bar X_m$ are definable subgroup of $(\RR,+)$ of bounded index (for $i \in I$). And clearly $(\RR,+)^{00}_A$ is the intersection of all of them, so $(\RR,+)^{00}_A = (\RR,+)^0_A$.  Thus, the second part of the proposition follows by Theorem  \ref{theorem: main theorem}. 
\end{proof}

\section{Conflict of interest and data availability statement}
The author states that there is no conflict of interest. Data sharing not applicable to this article as no datasets were generated or analysed during the current study.

\printbibliography
\nocite{*}

	\end{document}